\documentclass[12pt]{amsart}
\usepackage{amsmath}
\usepackage{amssymb}
\usepackage{amsfonts}
\usepackage{amsthm}
\usepackage{mathrsfs}
\usepackage{color}
\usepackage{graphicx}
\usepackage{caption}
\usepackage[all]{xy}
\theoremstyle{remark}
\newtheorem*{theorem*}{Theorem}
\newtheorem{theorem}{Theorem}[section]
\newtheorem{lemma}[theorem]{Lemma}

\newtheorem{proposition}[theorem]{Proposition}

\usepackage{graphicx}
\graphicspath{{images/}}

\theoremstyle{remark}
\newtheorem{definition}[theorem]{Definition}
\newtheorem{remark}[theorem]{Remark}

\newtheorem{example}[theorem]{Example}

\title{Crossingless sheaves and their classes in equivariant K-theory}
\author[G.~Dobrovolska]{Galyna Dobrovolska}
\address{Department of Mathematics, Ariel University, Ariel 40700 Israel}
\email{galdobr@gmail.com}
\begin{document}

\maketitle

\begin{abstract} We introduce crossingless sheaves in certain equivariant derived categories which are analogous to the Bezrukavnikov-Mirkovic exotic sheaves for two-block nilpotents. We calculate the classes of crossingless sheaves in equivariant K-theory of Cautis-Kamnitzer varieties.
\end{abstract}

\section{Introduction}

The exotic t-structure was defined by Bezrukavnikov and Mirkovic in \cite{bm} in order to study modular representation theory of Lie algebras. The exotic t-structure is defined using an action of the affine braid group on derived categories of varieties related to the Springer resolution defined by Bezrukavnikov and Riche in \cite{br}. The irreducible objects in the heart of an exotic t-structure are called simple exotic sheaves.


More precisely, let $\bf k$ be an algebraically closed field of characteristic $p \gg 0$ and let $\mathfrak g$ be a reductive Lie algebra defined over $\bf k$. Let $\lambda$ be a regular integral weight and let $e$ be a nilpotent. Let ${\rm Mod}_e^{fg,\lambda}$ be the category of modules with genralized central character $(\lambda, e)$. Let $\mathcal B_{e, \bf k}$ be the Springer fiber corresponding to $e$. Then Theorem 5.3.1 of \cite{bmr} states that
$$D^b(Coh_{\mathcal B_{e, \bf k}}(\widetilde{g}_{\bf k})) \simeq D^b({\rm Mod}_e^{fg,\lambda}).$$
It is also known that the tautological t-structure on the derived category of modules corresponds to the exotic t-structure on the derived category of coherent sheaves.

Exotic sheaves for two-block nilpotents were studied by Anno and Nandakumar in \cite{an}. Let $\mathcal{D}_n^{\prime}$ be the derived category of coherent sheaves on a two-block Springer fibre (where the nilpotent is in type A and has two blocks of sizes $m+n$ and $n$). Anno and Nandakumar defined the objects $\Psi^{\prime}_{\alpha} \in D_n^{\prime}$ for any $(m,n)$-crossingless matching $\alpha$ and proved that these objects are the simple exotic sheaves for two-block nilpotents. Cautis-Kamnitzer varieties $Y_k$ which contain Springer fibers and Slodowy slices were introduced in \cite{ck} (see Definition \ref{defofY_k}). The equivariant K-theory of $Y_k$ is a free $\mathbb{C}[q, q^{-1}]$-module on $2^k$ generators (See Proposition \ref{K_0ofY_k}). Classes of $\Psi^{\prime}_{\alpha}$ in K-theory of Cautis-Kamnitzer varieties were computed in \cite{dny} and subsequently applied to the calculation of dimensions of modular representations of $\mathfrak{sl}_n$ with a two-block central character.

Cautis and Kamnitzer studied categorification of Reshetikhin-Turaev invariants of linear tangles in \cite{ck}. In the course of their study they defined functors $\Psi(\alpha)$ for linear tangles $\alpha$. We complete their representation to a representation of affine tangles so that we have functors $\Psi(\alpha)$ for any affine tangle $\alpha$. Using these functors, we define an analogous concept to Anno-Nandakumar's exotic sheaves for a two-block nilpotent in certain equivariant derived category. We call the objects we defined {\it crossingless sheaves} (see Definition \ref{crossingless}). We compute the classes of crossingless sheaves in equivaraint K-theory of Cautis-Kamnitzer varieties (see Theorem \ref{maintheorem}, which is our main theorem).

This paper introduces techniques which can be used for defining representations of affine tangles on equivariant derived category of Cautis-Kamnitzer varieties with respect to the ${\mathbb C}^*$ in the centralizer of the two-block nilpotent and calculating characters of modular representations with central character equal to a two-block nilpotent (note that dimension formulas for these modular representations have been computed in \cite{dny}). These goals will be achieved in subsequent publications by the author.

\subsection*{Acknowledgements} I thank Vinoth Nandakumar for sharing his ideas and for many useful conversations. I also thank Joel Kamnitzer for useful correspondence and Roman Travkin for helpful conversations.

\section{Preliminaries} \label{prelim}

\subsection{Affine tangles} The combinatorics of affine tangles will be necessary for what follows; we recall the definitions here (see Section 3 of \cite{an} for more details). 

\begin{definition} If $q \equiv r \pmod 2$, a $(q,r)$ affine tangle is an embedding (up to isotopy) of $\frac{q+r}{2}$ arcs and a finite number of circles into the region $\{ (x, y) \in \mathbb{C} \times \mathbb{R} | 1 \leq |x| \leq 2 \}$, such that the end-points of the arcs are $(1,0), (\zeta_q, 0), \cdots , (\zeta_q^{q-1}, 0), (2, 0), (2\zeta_r,0), \cdots, (2 \zeta_r^{r-1},0)$ in some order; here $\zeta_k = e^{\frac{2\pi i}{k}}$. \end{definition}

Figure \ref{composition} illustrates the composition of affine tangles.

\begin{figure} 
\centering
\captionsetup{justification=centering}
\includegraphics[scale=1]{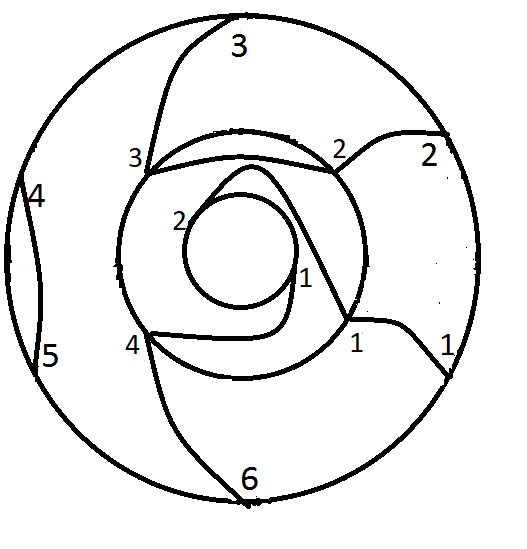}
\caption{}
\label{composition}
\end{figure}

\begin{definition} If $q \equiv r \pmod 2$, a framed $(q,r)$ affine tangle is an embedding (up to isotopy) of $\frac{q+r}{2}$ ``rectangular arcs" and a finite number of circles into the region $\{ (x, y) \in \mathbb{C} \times \mathbb{R} | 1 \leq |x| \leq 2 \}$. Here a ``rectangular arc" is an injective map from $[0,1] \times [0,1]$ to $\{ (x, y) \in \mathbb{C} \times \mathbb{R} | 1 \leq |x| \leq 2 \}$ where the segments $[0,1] \times 0$ and $[0,1] \times 1$ are mapped to the segments $[x] \times [0,1]$ where $x = \zeta_p^i$ or $x = 2 \zeta_q^j$. 
 \end{definition} 

\begin{remark} Framed tangles are necessary for the validity of the Reidemeister I move in the functorial tangle categorification statement from Proposition \ref{tanglecat}. Below we will often abbreviate ``$(q,r)$-affine tangle'' to ``$(q,r)$-tangle''. In \cite{an}, we used $\; \widehat{} \;$ notation to denote framed tangles; we omit that here. Note that given a $(q,r)$ framed affine tangle $\alpha$, and a $(r,s)$ framed affine tangle $\beta$, we can compose them using scaling and concatenation to obtain a $(q,s)$ affine tangle $\beta \circ \alpha$. This composition is associative up to isotopy. \end{remark}

\begin{definition} Given $1 \leq i \leq n$, define the following affine tangles: \begin{itemize} \item Let $g_n^i$ denote the $(n-2, n)$ tangle with an arc connecting $(2\zeta_n^i,0)$ to $(2\zeta_n^{i+1},0)$.
\item Let $f_n^i$ denote the $(n, n-2)$ tangle with an arc connecting $(\zeta_{n}^i,0)$ and $(\zeta_{n}^{i+1},0)$. 
\item Let $t_n^i(1)$ (respectively, $t^i_n(2)$) denote the $(n,n)$ tangle in which a strand connecting $(\zeta_n^{i},0)$ to $(2\zeta_{n}^{i+1},0)$ passes above (respectively, beneath) a strand connecting $(\zeta_n^{i+1},0)$ to $(2\zeta_n^{i},0)$. 
\item Let $r_n$ denote the $(n,n)$ tangle connecting $(\zeta_n^j, 0)$ to $(2\zeta_n^{j-1}, 0)$ for each $1 \leq j \leq n$ (clockwise rotation of all strands). 
\item Let $s_n^{i}$ denote the $(n,n)$-tangle with a strand connecting $(\zeta_j,0)$ to $(2 \zeta_j,0)$ for each $j$, and a strand connecting $(\zeta_i,0)$ to $(2 \zeta_i,0)$ passing clockwise around the circle, beneath all the other strands.
\end{itemize} 
\end{definition}

We also introduce new generators $w_n^i(1)$ and $w_n^i(2)$, which correspond to positive and negative twists of framing of the $i$-th strand of an $(n, n)$ identity tangle.

The below figures depict some of these elementary tangles. Here the label $i$ in the inner (resp. outer) circle denotes the point $(\zeta^i, 0)$ (resp. $(2 \zeta^i,0)$). 

\includegraphics[scale=0.65]{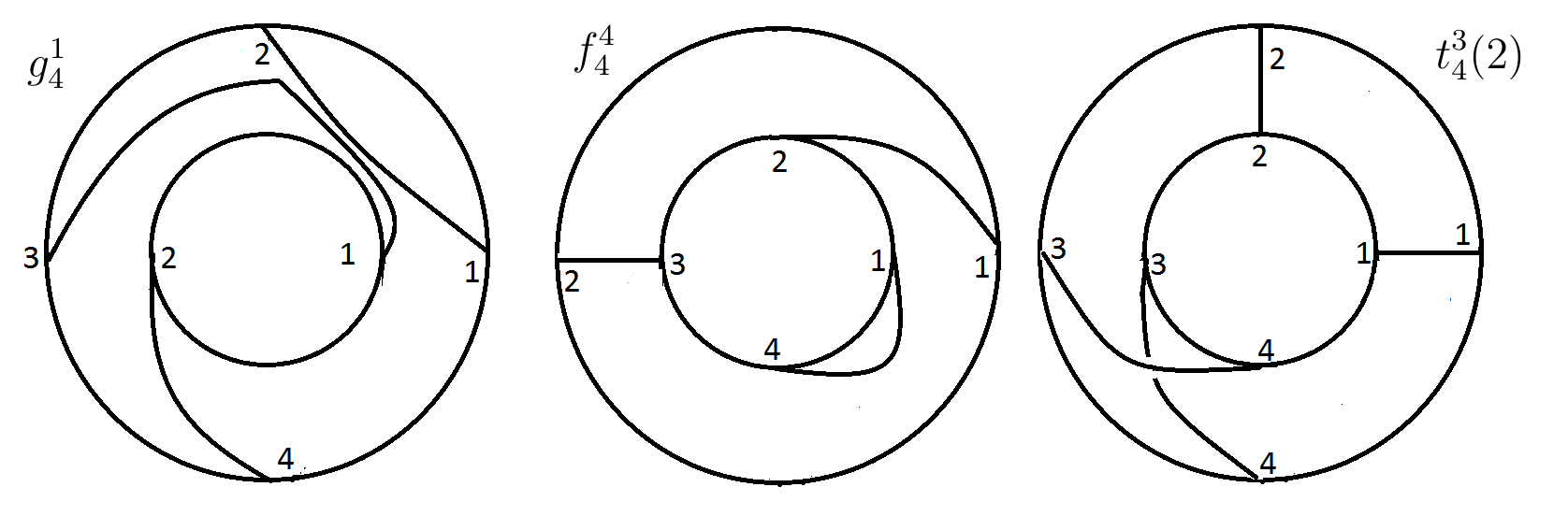}

\includegraphics[scale=0.65]{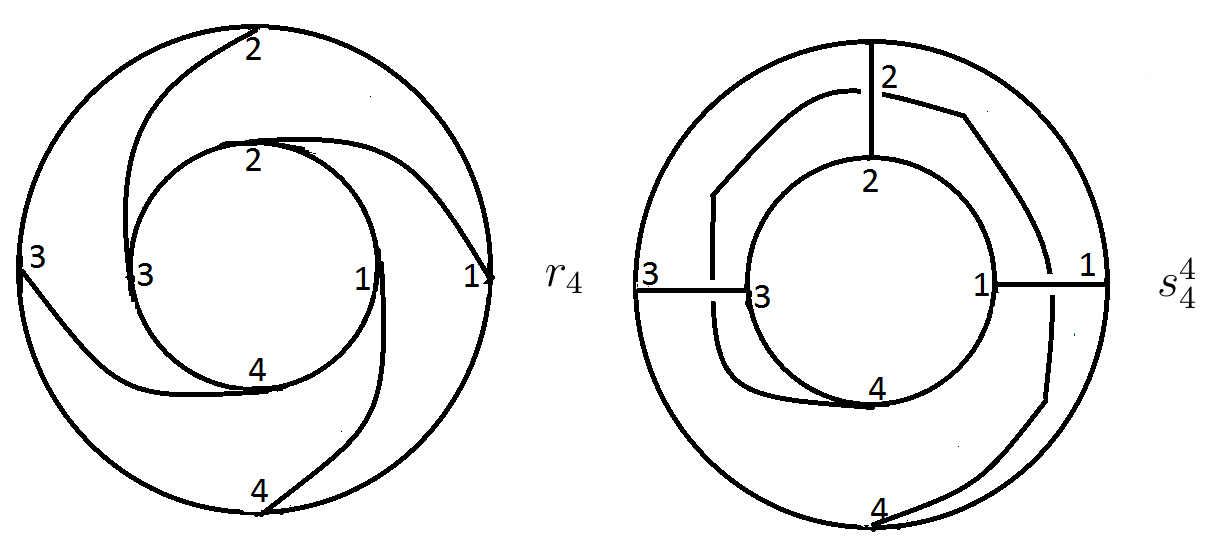}

These are called ``elementary'' tangles because any framed tangle can be expressed as (or rather, is isotopic to) a product of the tangles $g_n^i, , f_n^i, t_n^i(1), t_n^i(2), w_n^i(1), w_n^i(2)$ and $r_n$. Alternatively, this statement is true if we replace $r_n$ by $s_n^n$ (to see this, one uses the fact that $r_n = s_n^n \circ t_n^{n-1}(2) \circ \cdots \circ t_n^1(2)$). 

\begin{definition} A $(q,r)$-tangle is a``linear tangle'' if it can be expressed in terms of the generators $g_n^i, f_n^i, t_n^i(1), t_n^i(2), w_n^i(1), w_n^i(2)$ (i.e. without using $r_n$). \end{definition}

The above definition is motivated that the fact any such tangle can be represented with endpoints on a line segment (as opposed to an annulus); see Section 4 of \cite{ck} for a pictorial depiction.


Below is the list of relations which the generators $g_n^i, f_n^i, t_n^i(1), t_n^i(2), w_n^i(1), w_n^i(2), r_n$ satisfy. Note that $r^{\prime}_n$ is the inverse of $r_n$. 

\begin{enumerate} 
\item\label{AFTanMovesFirst} ${f}_n^i\circ {g}_n^{i+1} = id = {f}_n^{i+1} \circ {g}_n^i$ 
\item\label{AFTanReidemeister1} (Reidemeister 1) ${f}_n^i\circ {t}_n^{i\pm 1}(l)\circ {g}_n^i = {w}_n^i(l)$ 
\item\label{AFTanMovesThird} ${t}_n^i(2)\circ {t}_n^i(1) = id = {t}_n^i(1)\circ {t}_n^i(2)$ 
\item\label{AFTanReidemeister3} ${t}_n^i(l)\circ {t}_n^{i+1}(l)\circ {t}_n^i(l) =  {t}_n^{i+1}(l)\circ {t}_n^{i}(l)\circ {t}_n^{i+1}(l)$ 
\item ${g}_{n+2}^{i+k}\circ {g}_n^i = {g}_{n+2}^i\circ {g}_n^{i+k-2}$ 
\item ${f}_{n}^{i+k-2}\circ {f}_{n+2}^i = {f}_{n}^i\circ {f}_{n+2}^{i+k}$ 
\item ${g}_n^{i+k-2}\circ {f}_n^i = {f}_{n+2}^i\circ {g}_{n+2}^{i+k}, \quad {g}_n^{i}\circ {f}_n^{i+k-2} = {f}_{n+2}^{i+k}\circ {g}_{n+2}^{i}$
\item\label{AFTanPitchfork} ${g}_n^i\circ {t}_{n-2}^{i+k-2}(l) = {t}_n^{i+k}(l)\circ {g}_n^i, \quad {g}_n^{i+k}\circ {t}_{n-2}^{i}(l) = {t}_n^{i}(l)\circ {g}_n^{i+k}$ 
\item ${f}_n^i\circ {t}_{n}^{i+k}(l) = {t}_{n-2}^{i+k-2}(l)\circ {f}_n^i, \quad {f}_n^{i+k}\circ {t}_{n}^{i}(l) = {t}_{n-2}^{i}(l)\circ {f}_n^{i+k}$ 
\item ${t}_n^i(l)\circ {t}_n^{i+k}(m) =  {t}_n^{i+k}(m)\circ {t}_n^i(l)$ \item\label{TanMovesLast}  ${t}_n^i(1)\circ {g}_n^{i+1} = {t}_n^{i+1}(2)\circ {g}_n^i,\quad {t}_n^i(2)\circ {g}_n^{i+1}= {t}_n^{i+1}(1)\circ {g}_n^i$ 
\item ${r}_n\circ {r}_n'  = id = {r}_n' \circ {r}_n$ 
\item ${r}_{n-2}'\circ {f}_n^i \circ {r}_n  = {f}_n^{i+1},\ i=1,\ldots, n-2; \quad
 {f}_n^{n-1} \circ ({r}_n)^2  = {f}_n^1$ 
\item ${r}_n' \circ {g}_n^i \circ {r}_{n-2}  = {g}_n^{i+1}, \ i=1,\ldots, n-2; \quad
 ({r}_n')^{2} \circ {g}_n^{n-1} = {g}_n^1$ 
\item  ${r}_n' \circ {t}_n^i(l)\circ {r}_n  = {t}_n^{i+1}(l); \quad
 ({r}_n')^{2}\circ {t}_n^{n-1}(l) \circ ({r}_n)^2  = {t}_n^1(l)$ 

We have the following additional relations for twists:

\item\label{AFTanMovesFirstTwist} ${w}_n^i(1) \circ {w}_n^i(2) = id, \quad {w}_n^i(l) \circ {w}_n^j(k) = {w}_n^j(k) \circ {w}_n^i(l),\ i\ne j$ 
\item $\hat{w}_n^i(k) \circ {g}_n^i = {w}_n^{i+1}(k) \circ {g}_n^i, \quad%
 {w}_n^i(k) \circ {g}_n^j =  {g}_n^j \circ {w}_n^{i+1\pm 1}(k),\ i\ne j, j+1$
\item ${f}_n^i \circ {w}_n^i(k) = {f}_n^i \circ {w}_n^{i+1}(k), \quad%
{w}_n^i(k) \circ {f}_n^j =  {f}_n^j \circ {w}_n^{i-1\pm 1}(k),\ i\ne j, j+1$
\item ${w}_n^i(k) \circ {t}_n^i(l) = {w}_n^{i+1}(k) \circ {t}_n^i(l), \quad%
 {w}_n^i(k) \circ {t}_n^j(l) =  {t}_n^j(l) \circ {w}_n^{i}(k),\ i\ne j, j+1$ 
\item \label{FTanMovesLast} ${t}_n^i \circ {w}_n^i(k) = {f}_n^i \circ {w}_n^{i+1}(k), \quad%
{w}_n^i(k) \circ {f}_n^j =  {t}_n^j \circ {w}_n^{i}(k),\ i\ne j, j+1$
\item\label{AFTanMovesLast} ${w}_n^i(k) \circ {r}_n = {r}_n \circ {w}_n^{i-1}(k), \quad%
 {w}_n^i(k) \circ {r}'_n = {r}'_n \circ {w}_n^{i+1}(k)$ \end{enumerate}

\subsection{Exotic t-structures for two-block nilpotents} \label{[AN]}

In this section we summarize the results of \cite{an} for motivational purposes. Let $\mathfrak{g}=\mathfrak{sl}_{m+2n}$, and $\mathcal{B}$ be the flag variety. Let $e \in \mathfrak{g}^*$ be a nilpotent linear functional with Jordan type $(m+n, n)$, and $\lambda \in \mathfrak{h}^* // W$ be a regular Harish-Chandra character. 

Let $S_{m,n} \subset \mathfrak{g}$ be the corresponding Mirkovic-Vybornov transverse slice (see \cite{mv} for a definition, and see below for an explicit matrix realization). Let $U_{m+n,n} \subset \tilde{\mathfrak{g}} = T^* \mathcal{B}$ be its preimage under the Grothendieck-Springer resolution map.  

\begin{align*} S_{m,n} &= \{ \left( \begin{array}{ccccccccc}
 & 1 &  & & &  &  &  &  \\
 &  & & \ddots &  &  &  &  &  \\
 &  &  & & 1 &  &  &  &  \\
a_1 & a_2 & & \cdots & a_{m+n} & b_1 & b_2 & \cdots  & b_{n} \\
 &  &  &  & &  & 1 &  &  \\
 &  &  &  & &  &  & \ddots &  \\
 &  &  &  & &  &  &  & 1 \\
c_1 & \cdots & c_n \; 0 & \cdots & 0 & d_1 & d_2 & \cdots & d_n
\end{array} \right) \} \end{align*}

\begin{definition} Let $\langle \cdot, \cdot \rangle$ denote the Killing form. 
\begin{align*} S^*_{m,n} &= \{ x \in \mathfrak{g} \; | \; \langle x, z - e \rangle = 0 \; \forall \; z \in S_{m,n} \} \\ \textbf{U}_{\hat{e}} \mathfrak{g} &= \textbf{U} \mathfrak{g} / \langle x^p - x^{[p]} - \langle e,x \rangle^p \; | \; x \in S^*_{m,n} \rangle \end{align*}

Denote by $\text{Mod}_{\hat{e}, \lambda}^{fg}(\textbf{U}\mathfrak{g})$ the subcategory of $\text{Mod}_{e, \lambda}^{fg}(\textbf{U}\mathfrak{g})$ consisting of those modules which are defined over the quotient $\textbf{U}_{\hat{e}} \mathfrak{g}$. \end{definition}

[BMR] localization theory gives the equivalence in the below Proposition; see Section 5.1 of \cite{an} for more details. Under this equivalence, the tautological t-structure on the LHS is mapped to the exotic t-structure on the RHS. We refer to the images of the irreducible objects in the LHS as simple exotic sheaves.

\begin{proposition} \label{cmn} $$\mathcal{C}_{m,n} := D^b \text{Mod}^{\text{fg}}_{\hat{e}, \lambda}(U\mathfrak{g}) \simeq D^b \text{Coh}_{\mathcal{B}_e}(U_e)$$ \end{proposition}

\begin{proposition} \label{tanglefunctor} Let $\alpha$ be an $(m+2p, m+2q)$-tangle. Then we have a functor $\Psi^{\prime}(\alpha): \mathcal{C}_{m,p} \rightarrow \mathcal{C}_{m,q}$; this collection of functors satisfy the functorial tangle relations $\Psi^{\prime}(\alpha) \circ \Psi^{\prime}(\beta) \simeq \Psi^{\prime}(\alpha \circ \beta)$. \end{proposition}

These are defined by Fourier-Mukai transforms on the categories of coherent sheaves; see Section $4$ of \cite{an} for more details of the proof. These functors can be used to describe the simple exotic sheaves as follows.  

\begin{definition} \label{matching} Let $\text{Cross}(m,n)$ be the set of all matchings on an annulus with $m$ unlabelled points on the inner circle, $m+2n$ labelled points on the outer circle (with labels $1, \cdots, m+2n$), such that there are no crossings or loops, and each point on the inner circle is connected to a point on the outer circle. \end{definition}

\begin{example} \label{cross22} See Figure \ref{example} for an example of an element $\alpha \in \text{Cross}(2,2)$. \end{example}

\begin{figure} 
\centering
\captionsetup{justification=centering}
\includegraphics[scale=0.75]{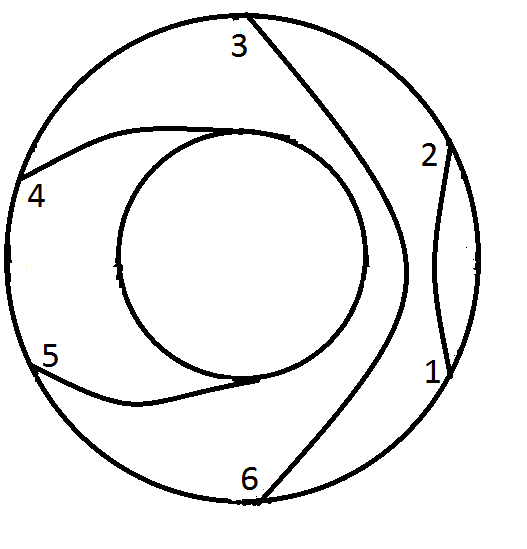}
\caption{}
\label{example}
\end{figure}

\begin{definition} \label{defofPsiprime} Each $\alpha \in \text{Cross}(m,n)$ can be viewed as a framed tangle, with the trivial framing. Thus using Proposition \ref{tanglefunctor} we have a functor $\Psi^{\prime}(\alpha): \mathcal{C}_{m,0} \rightarrow \mathcal{C}_{m,n}$. Since $\mathcal{B}_e$ is a point when $n=0$, it follows that $\mathcal{C}_{m,0} \simeq D^b(\text{Vect})$.Given $\alpha \in \text{Cross}(m,n)$, let $\Psi^{\prime}_{\alpha} = \Psi^{\prime}(\alpha)(\textbf{k})$. \end{definition}
\begin{theorem} \label{exotic} The objects $\{ \Psi^{\prime}_{\alpha} \; | \; \alpha \in \text{Cross}(m,n) \}$ are the simple exotic sheaves. \end{theorem} 
See Section 5.3 of \cite{an} for a proof of the above theorem. 

\section{Coherent sheaves on Cautis-Kamnitzer varieties and the main theorem} \label{cohsheaves}



We will need Cautis-Kamnitzer's framework for much of what follows. The Cautis-Kamnitzer varieties $Y_k$ are constructed as follows (see Section 2.1 of \cite{ck} for more details):

\begin{definition} \label{defofY_k} Consider a $2k$-dimensional vector space $V_k$ over $\mathbb C$ with basis $e_1, \ldots , e_{k}, f_1, \ldots ,f_k$ and a nilpotent $z$ such that $z e_{i+1} = e_{i}$, $z f_{i+1} = f_{i}$ for $i \geq 1$, and $ze_1 = zf_1=0$. Let $L_0=0$. Then:  \begin{align*} Y_{k} = \{ (L_1 \subset \cdots \subset L_{k} \subset V_{k}) \; | \; \text{ dim}(L_i) = i,\ zL_{i} \subset L_{i-1} \} \end{align*} 

Let $\mathbb C^*$ act on $V_k$ by $t \cdot e_i = t^{-2i} e_i$ and $t \cdot f_i = t^{-2i} f_i$. This $\mathbb C^*$-action induces a $\mathbb C^*$-action on $Y_k$ by $t \cdot (L_1,...,L_k) = (t \cdot L_1,...,t \cdot L_k)$. Denote by $K^0_{\mathbb C^*}(Y_{k})$ the equivariant $K$-theory of $Y_k$ with respect to the above action of $\mathbb C^*$.
\end{definition}

We have a natural forgetful map $p_k: Y_k \rightarrow Y_{k-1}$ which is a $\mathbb{P}^1$-bundle. To see this, suppose that we have $(L_1,..., L_{k-1}) \in Y_{k-1}$ and are considering possible choices of $L_k$. It
is easy to see that we must have $L_{k-1} \subset L_k \subset z^{-1}(L_k)$. Since $z^{-1}(L_{k-1})/L_{k-1}$ is always two dimensional,
this fibre is a ${\mathbb P}^1$. It follows that $Y_k$ is an iterated $\mathbb{P}^1$-bundle of dimension $k$, since $Y_1 \simeq \mathbb{P}^1$.

\begin{definition} Let $\mathcal{V}_s$ be the tautological vector bundle on $Y_k$ corresponding to $L_s$. For each $1 \leq s \leq k$, we have a line bundle $\Lambda_s = \mathcal{V}_s/\mathcal{V}_{s-1}$ on this space (note that non-equivariantly we have $\Lambda_k = \mathcal{O}(-1)_{Y_k / Y_{k-1}}$, i.e. the relative $\mathcal{O}(-1)$ on the $\mathbb{P}^1$-fibration $Y_k \rightarrow Y_{k-1}$). Given $i_s \in \{ 0,1 \}$ for each $1 \leq s \leq k$, define also the following line bundle: $$\Lambda_{i_1, i_2, \cdots, i_k} := \bigotimes_{1 \leq s \leq k} \Lambda_{s}^{i_s}$$ \end{definition} 

Recall from Section $6.2$ of \cite{ck} that:
\begin{proposition} \label{K_0ofY_k} We have that $K^0_{\mathbb C^*}(Y_{k})$ is a free $\mathbb{C}[q, q^{-1}]$-module with generators given by the classes of the line bundles $[\Lambda_{i_1, i_2, \cdots, i_k}]$, where $i_s \in \{ 0, 1 \}$ for each $1 \leq s \leq k$. \end{proposition}

\begin{definition} \label{def:K^0} 
Let $V$ denote the standard representation of $U_q(\mathfrak{sl}_2)$. It is a free ${\mathbb C}[q, q^{-1}]$-module with basis $v_0$, $v_1$.
Identify $K^0_{\mathbb C^*}(Y_{k})$ with $V^{\otimes k}$ as $\mathbb{C}[q, q^{-1}]$-modules, so that $$ [\Lambda_{i_1, \cdots, i_{k}}] = q^{-\sum_s s i_s}v_{i_1} \otimes \cdots \otimes v_{i_k}$$ \end{definition} \begin{definition} Denote the monomial basis for $V^{\otimes k}$ as follows. Let $I \subset \{ 1, \cdots, k \}$ be a subset, and for $1 \leq a \leq n$ let $I(a) = 1$ if $a \in I$, and $I(a) = 0$ if $a \notin I$. Define: $$v_I = v_{I(1)} \otimes v_{I(2)} \otimes \cdots \otimes v_{I(k)} $$ 
\end{definition}

\begin{definition} Let $\mathcal{D}_k = D^b_{\mathbb C^*}\text{Coh}(Y_k)$ be the equivariant derived category of coherent sheaves on $Y_k$. \end{definition}

\begin{definition} \label{def_of_fgt}  
For each linear $(r,s)$-tangle $\alpha$, we define the map $\psi(\alpha): V^{\otimes r} \rightarrow V^{\otimes s}$ of $U_q({\mathfrak sl}_2)$-modules as follows. It suffices to define the map on the elementary linear tangles. Below $k \in \{ 1, 2 \}$ and $i,j \in \{0, 1 \}$.
\begin{align*}\psi(f_2^1): & V^{\otimes 2} \rightarrow \mathbb{C}[q, q^{-1}]; \qquad \psi(f_2^1)(v_i \otimes v_i) = 0, \psi(f_2^1)(v_0 \otimes v_1) = q, \psi(f_2^1)(v_1 \otimes v_0) = -1. \\
\psi(g_2^1): & \mathbb{C}[q, q^{-1}] \rightarrow V^{\otimes 2}; \qquad \psi(g_2^1)(1) = q^{-1} v_1 \otimes v_0 - v_0 \otimes v_1. \\
\psi(t_2^1(1)): & V^{\otimes 2} \to V^{\otimes 2}; \qquad \psi(t_2^1(1))(v_0 \otimes v_0) = v_0 \otimes v_0, \psi(t_2^1(1))(v_0 \otimes v_1) = (1-q^2)v_0 \otimes v_1 + q v_1 \otimes v_0, \\
\ \ \ \ \ \ \ \ \ \ \ \ \ \ \  & \psi(t_2^1(1))(v_1 \otimes v_0) = v_0 \otimes v_1, \psi(t_2^1(1))(v_1 \otimes v_1) = v_1 \otimes v_1. \\ 
\psi(t_2^1(2)): & V^{\otimes 2} \to V^{\otimes 2}; \qquad \psi(t_2^1(2))(v_0 \otimes v_0) = v_0 \otimes v_0, \psi(t_2^1(2))(v_0 \otimes v_1) = q^{-1}v_1 \otimes v_0, \\
\ \ \ \ \ \ \ \ \ \ \ \ \ \ \  & \psi(t_2^1(2))(v_1 \otimes v_0) = (1-q^{-2})v_1 \otimes v_0 + q^{-1} v_0 \otimes v_1, \psi(t_2^1(2))(v_1 \otimes v_1) = v_1 \otimes v_1. \end{align*}
Now we define $\psi(g_n^i)$ as follows, and similarly for $\psi(f_n^i)$ and $\psi(t_n^i(k))$. 
\begin{align*} \psi(g_n^i): & V^{\otimes n-2} \rightarrow V^{\otimes n};  \psi(g_n^i) = \text{id}_{V^{\otimes i-1}} \otimes \psi(g_2^1) \otimes  \text{id}_{V^{\otimes n-2-i+1}}  \end{align*}
Note that $\psi$ is well-defined due to \cite{rt} and the introduction to \cite{ck}.
\end{definition} 


\begin{remark} We can identify $\psi(\alpha): V^{\otimes r} \rightarrow V^{\otimes s}$ with the Reshetikhin-Turaev invariant of $\alpha$ from \cite{rt} up to a change of notation.
\end{remark}

The below proposition is a strengthening of the results in Section 6 of \cite{ck}, the difference being that we use affine tangles (and not linear tangles). 

\begin{proposition} \label{tanglecat} For each $(q,r)$-tangle $\alpha$, there exist a functor $\Psi(\alpha): \mathcal{D}_q \rightarrow \mathcal{D}_r$, and these are compatible under composition, i.e. given an $(r,s)$-angle $\beta$, $\Psi(\alpha \circ \beta) \simeq \Psi(\alpha) \circ \Psi(\beta)$. When $\alpha$ is a linear tangle, the image of $\Psi(\alpha)$ in the Grothendieck group is the Reshetikhin-Turaev invariant $\psi(\alpha)$. \end{proposition} \begin{proof} 
When $\alpha$ is a linear tangle, the functors are certain Fourier-Mukai kernels which are described in Section 4 of \cite{an}; the functorial relations are proven in Section 5. Define the functors $\Psi(t^i_{m+2n}(1))$ and $\Psi(t^i_{m+2n}(2))$ via the distinguished triangles:
$$\Psi(g^i_{m+2n}) \Psi(f^i_{m+2n})[-1]\{1\} \to \text{id} \to \Psi(t^i_{m+2n}(1)),$$ 
$$\Psi(t^i_{m+2n}(2)) \to \text{id} \to \Psi(g^i_{m+2n})\Psi(f^i_{m+2n})[1]\{-1\}.$$ 
We have $\psi(t_2^1(1))=1+q\psi(g_2^1)\psi(f_2^1)$ and 
$\psi(t_2^1(2))=1+q^{-1}\psi(g_2^1)\psi(f_2^1)$. From these formulas the above formulas for $\psi(t_2^1(1)), \psi(t_2^1(2))$ follow. 
Theorem 6.5 of \cite{ck} proves that, on the level of Grothendieck groups, $\Psi(\alpha)$ corresponds to $\psi(\alpha)$.
It remains to construct the functors $\Psi(\alpha)$ when $\alpha$ is an affine tangle; this follows from the techniques used in Section 4.5 of \cite{an}. It suffices to construct a functor $\Psi(s_{k}^k)$, and check that it satisfies the necessary functorial tangle relations (13) and (14) on page 7. Define $\Psi(s_{k}^k)$ to be the functor of tensoring with the line bundle $\Lambda_k^{-1}$; to check the functorial tangle relations, the same argument used in Theorem 4.27 and Proposition 4.25 of \cite{an} works verbatim. 

Alternatively, see Proposition 6.4 in \cite{ckquantumgeometricsatake}
\end{proof}

\begin{definition} Let $W_{m,n} \subset V_{m+2n}$ denote the vector subspace with basis $e_1, \ldots, e_{m+n}, f_1, \cdots, f_n$, so that $z|_{W_{m,n}}$ has Jordan type $(m+n,n)$. Identify $\mathfrak{sl}_{m+2n}$ with $\mathfrak{sl}(W_{m,n})$.
Denote by $\mathcal{B}_{m+n,n}$ the Springer fiber for $z \in \mathfrak{sl}(W_{m,n})$: $$ \mathcal{B}_{m+n,n} = \{ (0 \subset U_1 \subset \cdots \subset U_{m+2n-1} \subset U_{m+2n} = W_{m,n}) \; | \; z U_i \subseteq U_{i-1} \} $$ \end{definition}

By analogy with Definition \ref{defofPsiprime} (cf. Theorem \ref{exotic}) we define:

\begin{definition} \label{crossingless} For any $\alpha \in \text{Cross}(m,n)$, the {\it crossingless sheaf} $\Psi_\alpha$ is defined as the object $\Psi({\mathcal O}_{\mathcal{B}_{m,0}})$.
\end{definition}

Note that $\mathcal{B}_{m,0}$ is a point and $\Psi({\mathcal O}_{\mathcal{B}_{m,0}})$ is the skyscraper sheaf supported on the point $\mathcal{B}_{m,0}$ in $Y_m$.

In the remaining part of the paper we compute the classes of $\Psi_\alpha$ in equivariant K-theory of Cautis-Kamnitzer varieties.

For our computations, we will need the following lemma which tells us what the effect of tensoring by $\Lambda_i^{-1}$ on $K^0(Y_k)$ is, when $i \leq k$. It more or less follows from the argument used to prove Theorem $6.2$ of \cite{ck} (see equation (20) in particular). 
\begin{lemma} \label{inverse} We have the following identity. $$[\Lambda_{i}^{-1}] = 2 q^{-2}[\mathcal{O}] - q^{-4}[\Lambda_i]$$ \end{lemma} \begin{proof} Consider the two exact sequences below. Note here that $z^{-1} \mathcal{V}_{i-1}$ denotes the vector bundle whose fiber at a point is the pre-image of the corresponding fiber of $\mathcal{V}_{i-1}$ at that point, under the map $z$.  \begin{align*} 0 \rightarrow \text{ker}(z) \rightarrow z^{-1} & \mathcal{V}_{i-1} \rightarrow \mathcal{V}_{i-1} \rightarrow 0 \\ 0 \rightarrow \mathcal{V}_i / \mathcal{V}_{i-1} \rightarrow z^{-1} \mathcal{V}_{i-1} & / \mathcal{V}_{i-1} \rightarrow z^{-1} \mathcal{V}_{i-1} / \mathcal{V}_i \rightarrow 0 \end{align*} Taking determinants, we obtain the following (noting that $\text{ker}(z) \simeq \mathcal{O}^{\oplus 2}\{2\}, \text{det}(\text{ker}(z)) \simeq \mathcal{O}\{4\}$): 
\begin{align*} \text{det}(z^{-1} \mathcal{V}_{i-1}) &\simeq \text{det}(\mathcal{V}_{i-1}) \otimes \text{det}( \mathcal{O}^{\oplus 2}\{2\}) \\ \Lambda_i \otimes z^{-1} \mathcal{V}_{i-1} / \mathcal{V}_i &\simeq \text{det}(z^{-1} \mathcal{V}_{i-1} / \mathcal{V}_{i-1}) \\ &\simeq \text{det}(z^{-1} \mathcal{V}_{i-1}) \otimes {\text{det}(\mathcal{V}_{i-1})}^{-1} \simeq \text{det}( \mathcal{O}^{\oplus 2}\{2\}) \\ z^{-1} \mathcal{V}_{i-1} / \mathcal{V}_i &\simeq \Lambda_i^{-1} \otimes \text{det}( \mathcal{O}^{\oplus 2}\{2\}) \end{align*} Using the first exact sequence, $[z^{-1} \mathcal{V}_{i-1} / \mathcal{V}_{i-1}] = 2 [\mathcal{O}\{2\}]$; so $[\Lambda_{i}^{-1} \otimes \text{det}( \mathcal{O}^{\oplus 2}\{2\})] = 2 [\mathcal{O}\{2\}] - [\Lambda_i]$. Therefore 
$[\Lambda_{i}^{-1} \otimes \mathcal{O}\{4\}] = 2 [\mathcal{O}\{2\}] - [\Lambda_i]$, or $q^4[\Lambda_{i}^{-1}]=2q^2[\mathcal{O}]-[\Lambda_i]$. \end{proof}

\begin{definition} Given an $(m+2p, m+2q)$-tangle $\alpha$, denote by $\overline{\alpha}: (\mathbb{C}^2)^{\otimes m+2p} \rightarrow (\mathbb{C}^2)^{\otimes m+2q}$ the map induced on the Grothendieck group by the functor $\Psi(\alpha): \mathcal{D}_{m+2p} \rightarrow \mathcal{D}_{m+2q}$. Given an element $\alpha \in \text{Cross}(m,n)$, denote by $\underline{\alpha} \in (\mathbb{C}^2)^{\otimes m+2n}$ the image of the irreducible object $\Psi_{\alpha}$ in the Grothendieck group. 
\end{definition} 

\begin{lemma} \label{computation} Let $i_s \in \{0,1\}$ and $l_s \in \mathbb Z$ for $1 \leq s \leq k$. Then we have:
\begin{align*}  \overline{r_k}(v_0 \otimes v_{i_2} \otimes \cdots \otimes v_{i_k}) &= (1-q^2) (\sum_{2 \leq j \leq k} q^{\{\# s<j, i_s=1\}} q^k \delta_{i_j,1} v_{i_2} \otimes \cdots \otimes v_{i_{j-1}} \otimes v_0 \otimes v_{i_{j+1}} \otimes \cdots v_{i_k} \otimes v_0) + \\& \ \  \ \ \ 
+ q^{\{\# s, i_s=1\}} v_{i_2} \otimes \cdots \otimes v_{i_k} \otimes (2q^{-2}v_0 - q^{-4} q^{-k} v_1)\\
\overline{r_k}(v_1 \otimes v_{i_2} \otimes \cdots \otimes v_{i_k}) &= q^k v_{i_2} \otimes \cdots \otimes v_{i_k} \otimes v_0  \\
[\Lambda_1^{l_1} \otimes \Lambda_2^{l_2} \otimes \cdots \otimes \Lambda_k^{l_k}] &= \bigotimes_{s=1}^k (l_s q^{2l_s-2}q^{-s}v_1 - (l_s-1)q^{2l_s} v_0)
\end{align*} \end{lemma}

\begin{proof}
Using Proposition \ref{tanglecat}, we compute that: \begin{align*} \overline{r_k}(v_{i_1} \otimes \cdots \otimes v_{i_k}) &= \overline{s_k^k} \circ \overline{t_k^{k-1}(1)} \circ \cdots \circ \overline{t_k^1(1)} (v_{i_1} \otimes \cdots \otimes v_{i_k}) 
\end{align*}

Using Definition \ref{def_of_fgt} and straightforward induction, one shows that
$$  \overline{t_k^{k-1}(1)} \circ \cdots \circ \overline{t_k^1(1)}(v_0 \otimes v_{i_2} \otimes \cdots \otimes v_{i_k}) = $$
$$= (1-q^2) (\sum_{2 \leq j \leq k} q^{\{\# s<j, i_s=1\}} \delta_{i_j,1} v_{i_2} \otimes \cdots \otimes v_{i_{j-1}} \otimes v_0 \otimes v_{i_{j+1}} \otimes \cdots v_{i_k} \otimes v_1) 
+ q^{\{\# s, i_s=1\}} v_{i_2} \otimes \cdots \otimes v_{i_k} \otimes v_0,$$
$$\overline{t_k^{k-1}(1)} \circ \cdots \circ \overline{t_k^1(1)}(v_1 \otimes v_{i_2} \otimes \cdots \otimes v_{i_k}) = v_{i_2} \otimes \cdots \otimes v_{i_k} \otimes v_1 
$$
The conclusion follows by noting that for any $j_s \in \{0,1\}, s=1,\dots,k$
$$\overline{s_k^k}([\Lambda_1^{j_1}]\otimes \cdots \otimes [\Lambda_k^{j_k}])
=[\Lambda_1^{j_1}]\otimes \cdots \otimes [\Lambda_k^{j_k-1}],$$
and using Definition \ref{def:K^0}.

It follows from Lemma \ref{inverse} that $[\Lambda_i^l] = l q^{2l-2}[\Lambda_i] - (l-1)q^{2l} [\mathcal{O}]$ for all $l \in \mathbb{Z}$. From Definition \ref{def:K^0}, it follows that: \begin{align*} [\Lambda_1^{l_1} \otimes \Lambda_2^{l_2} \otimes \cdots \otimes \Lambda_k^{l_k}] &= [\Lambda_1^{l_1}] \otimes [\Lambda_2^{l_2}] \otimes \cdots \otimes [\Lambda_k^{l_k}] \\ &= (l_1 q^{2l_1-2}[\Lambda_1] - (l_1-1)q^{2l_1} [\mathcal{O}]) \otimes \cdots (l_k q^{2l_k-2}[\Lambda_k] - (l_k-1)q^{2l_k} [\mathcal{O}]) \\ &= \bigotimes_{s=1}^k (l_s q^{2l_s-2}q^{-s}v_1 - (l_s-1)q^{2l_s} v_0). \end{align*} 
Note that in the last equality we used Definition \ref{def:K^0}.
\end{proof} 

\begin{definition} \label{matrix} For any $k$ denote by $A_{r,k}$ the $2^k \times 2^k$ matrix which is the matrix of the linear operator $\overline{r_k}$ in the basis $v_I$, $I \subset \{1,\cdots,k\}$, i.e. such that $A_{r,k} v_I = \overline{r_k} v_I$. Note that the matrix $A_{r,k}$ has been calculated in Lemma \ref{computation}.
\end{definition}



\begin{definition} \label{Calpha} Given $\alpha \in \text{Cross}(m,n)$, define the set $C(\alpha)$ to consist of the following pairs: $$ C(\alpha) = \{ (i, j) \; | \; 1 \leq i < j \leq m+2n; i \text{ and } j \text{ are connected in the outer circle on the diagram } \alpha \} $$ \end{definition}

\begin{definition} \label{defsgnqD} Let $T(\alpha)$ consist of all $n$-element subsets $S \subset \{ 1, \cdots, m+2n \}$ such that $|S \cap C| = 1$ for each $C \in C(\alpha)$; clearly $|T(\alpha)| = 2^n$. Let $\text{max}(C)$ be the maximal element in $C$, and for each $S \in T(\alpha)$ define: $$\text{sgn}_q(S) = (-q)^{c(S)} \text{ where } c(S) = \# \{ C : C \in C(\alpha), |C| = 2, \; S \cap C= \text{max}(C) \}.$$ 
Let $D(\alpha)$ be the set of points on the outer circle which are connected to points on the inner circle by arcs. Let $D(\alpha)=\{d_1,\dots,d_m\}$ where $d_1 < \cdots < d_m$. For a subset $J \subset D(\alpha)$ and $i \in \{1,\dots,m\}$ we define $J_i=1$ if $d_i \in J$ and $J_i=0$ if $d_i \notin J$.
\end{definition}

\begin{example} Let $\alpha \in \text{Cross}(2,2)$ be as in Example \ref{cross22}. Then we have:
\begin{align*}  T(\alpha) &= \{ (1,3), (1,6), (2,3), (2,6) \} \\ \text{sgn}(1,3)&=1, \text{sgn}(2,6)=q^2, \text{sgn}(1,6)=\text{sgn}(2,3)=-q \end{align*} \end{example}

\begin{lemma} \label{base} Let $\alpha \in \text{Cross}(m,0)$ (ie. $\alpha$ is the ``identity'' $(m,m)$-tangle). Then: $$\underline{\alpha} = (v_0 - q^{-1} v_1) \otimes \cdots \otimes (v_0 - q^{-m} v_1) = \bigotimes_{i=1}^m (v_0 - q^{-i} v_1) \in V^{\otimes m} $$ \end{lemma} \begin{proof} The Springer fiber $y_m=\mathcal{B}_{m, 0}$ is a point; we will prove the statement that the class $[\mathcal{O}_{y_m}] \in K^0(Y_m)$ is $(v_0 - q^{-1}v_1) \otimes \cdots \otimes (v_0 - q^{-m}v_1)$ by induction. First note that when $m=1$, $Y_1 \simeq \mathbb{P}^1$, we have the following equivariant short exact sequence: \begin{align*} 0 \rightarrow \Lambda_{1} \rightarrow &\mathcal{O} \rightarrow i_* \mathcal{O}_{y_1} \rightarrow 0, \end{align*} 
and hence $\underline{\alpha}=[\Lambda_1^0] - [\Lambda_1^1] = v_0 - q^{-1}v_1$.

For the induction step, consider the $\mathbb{P}^1$-bundle $p_m: Y_m \rightarrow Y_{m-1}$. We have $p_m(y_m) = y_{m-1}$, and for every $m$ denote the inclusion of the point by $j_m: \{y_m\} \rightarrow Y_m$. Taking a relative version of the above exact sequence on the fiber gives: \begin{align*} 0 \rightarrow& \; \mathcal{O}_{p_m^{-1}(y_{m-1})} \otimes \Lambda_{0,\cdots,0,1} \rightarrow \mathcal{O}_{p_m^{-1}(y_{m-1})} \rightarrow j_{m*} \mathcal{O}_{y_m} \rightarrow 0 \end{align*} By the inductive hypothesis, $[\mathcal{O}_{y_{m-1}}]=(v_0 - q^{-1} v_1) \otimes \cdots \otimes (v_0 - q^{-m+1}v_1)$. Since $p_m^* \Lambda_{i_1, \cdots, i_{m-1}} = \Lambda_{i_1, \cdots, i_{m-1}, 0}$, using Definition \ref{def:K^0} we compute that: \begin{align*} [p_m^*]v_{i_1} \otimes \cdots \otimes v_{i_{m-1}} &= v_{i_1} \otimes \cdots \otimes v_{i_{m-1}} \otimes v_0 \\ [\mathcal{O}_{p_m^{-1}(y_{m-1})}] &= [p_m^* \mathcal{O}_{y_{m-1}}] = (v_0 -q^{-1}) v_1 \otimes \cdots \otimes (v_0 - q^{-m+1}v_1) \otimes v_0.  \end{align*} We also deduce that $[\mathcal{O}_{p_m^{-1}(y_{m-1})} \otimes \Lambda_{0,\cdots,0,1}] =(v_0 -q^{-1}) v_1 \otimes \cdots \otimes (v_0 - q^{-m+1}v_1) \otimes q^{-m} v_1$ (using Definition \ref{def:K^0}). This now completes the proof: $$[j_{m*} \mathcal{O}_{y_m}] = [\mathcal{O}_{p_m^{-1}(y_{m-1})}] - [\mathcal{O}_{p_m^{-1}(y_{m-1})} \otimes \Lambda_{0,\cdots,0,1}] = \bigotimes_{i=1}^m (v_0 - q^{-i} v_1).$$ \end{proof} 

\begin{definition} We say that a matching $\alpha \in \text{Cross}(m,n)$ is ``good'' if $C_1(\alpha) = \emptyset$. \end{definition}

\begin{lemma} \label{class_easy} If the matching $\alpha$ is good, then in the notation of Definition \ref{defsgnqD} we have: \begin{align*} \underline{\alpha} &= q^{-n} \sum_{\substack{I \cup J, \\ I \in T(\alpha), \\ J \subset D(\alpha)}} \text{sgn}_q(I) (-q)^{\sum i J_i} v_{I \cup J} \qquad \qquad \end{align*} \end{lemma}

\begin{proof}[Proof of Lemma \ref{class_easy}] Proof by induction. For the induction step, note that if $\alpha \in \text{Cross}(m,n)$ is a good matching, we can pick $i$ such that $1 < i <m+2n$ and $(i, i+1) \in C(\alpha)$, express $\alpha = g_{m+2n}^i \beta$, for some $\beta \in \text{Cross}(m,n-1)$. Then we can complete the proof using notation from Definition \ref{defsgnqD} and Proposition \ref{tanglecat}: \begin{align*} \underline{\alpha} &= \psi(g_{m+2n}^i) \underline{\beta} \\ &= q^{-n+1} \sum_{\substack{I \cup J, \\ I \in T(\beta), \\ J \subset D(\beta)}}  (-q)^{\sum i J_i} \text{sgn}_q(I) \psi(g_{m+2n}^i) v_{I \cup J}  \\ &= q^{-n} \sum_{\substack{I \cup J, \\ I \in T(\alpha), \\ J \subset D(\alpha)}} (-q)^{\sum i J_i} \text{sgn}_q(I) v_{I \cup J} \end{align*} \end{proof}  

\begin{definition} \label{vector} For a good matching $\alpha \in {\rm Cross}(m,n)$, denote by $b_{\alpha} = \sum b_{\alpha, I} v_I$ the vector in the vector space spanned by $v_I$, $I \in\{1,\cdots,m+2n\}$ such that $b_{\alpha}=\underline{\alpha}$. Note that for any good matching $\alpha$, the vector $b_{\alpha}$ was calculated in Lemma \ref{class_easy}.
\end{definition}

Let $\alpha = r_{m+2n}^k \beta$ where $\beta$ is a good matching and $k \geq 0$ (to see the existence of such an expression, simply consider the arc $(i,j) \in C(\alpha)$ with maximal distance between its endpoints; then $k$ corresponds to one of the endpoints). Recalling Definitions \ref{matrix} and \ref{vector}, we have:

\begin{theorem} \label{maintheorem} For a crossingless matching $\alpha \in {\rm Cross}(m,n)$, the class in equivariant K-theory of $Y_{m+2n}$ of the sheaf $\Psi_{\alpha}$ is equal to $A_{r,m+2n}^k b_{\beta}$. 
\end{theorem}

\bibliographystyle{plain}

\end{document}